\documentclass{article}
\usepackage[
journal=JNCG,
lang=british,
]{ems-journal}
\makeatletter
\renewcommand*\ps@titlepage{%
  \let\@oddfoot\@empty
  \let\@evenfoot\@empty
  \let\@oddhead\@empty
  \let\@evenhead\@empty
}
\makeatother
\usepackage{tikz-cd}      
\usetikzlibrary{matrix}

\theoremstyle{plain}
\newtheorem{theorem}{Theorem}[section]
\newtheorem{conj}[theorem]{Conjecture}

\newtheorem{question}[theorem]{Question}
\newtheorem{lemma}[theorem]{Lemma}
\newtheorem{proposition}[theorem]{Proposition}

\theoremstyle{definition}
\newtheorem{definition}[theorem]{Definition}

\theoremstyle{remark}
\newtheorem{remark}[theorem]{Remark}

\numberwithin{equation}{section}

\newcommand{\ZZ}{\mathbb{Z}}

\newcommand{\cC}{\mathcal{C}}

\newcommand{\cB}{\mathcal{B}}

\numberwithin{equation}{section}

\begin{document}





\emsauthor{2}{
	\givenname{Yeqin}
	\surname{Liu}
	\mrid{1591694}
	\orcid{0000-0001-8231-7682}}{Y.~Liu}
\emsauthor{3}{
	\givenname{Yu}
	\surname{Shen}
	\mrid{}
	\orcid{0000-0001-5766-1596}}{Y.~Shen}
\emsauthor{4}{
	\givenname{Ziyu}
	\surname{Zhang}
	\mrid{}
	\orcid{0000-0001-8692-4842}}{Z. ~Zhang}


\Emsaffil{2}{
  \department{Department of Mathematics}
  \organisation{University of Michigan}
  \rorid{00jmfr291}
  \address{530 Church St}
  \zip{48109}
  \city{Ann Arbor, MI}
  \country{USA}
  \affemail{yqnl@umich.edu}
}

\Emsaffil{3}{
  \department{Department of Mathematics}
  \organisation{Michigan State University}
  \rorid{05hs6h993}
  \address{619 Red Cedar Road}
  \zip{48824}
  \city{East Lansing, MI}
  \country{USA}
  \affemail{shenyu5@msu.edu}
}

\Emsaffil{4}{
  \department{Insitute of Mathematical Science}
  \organisation{ShanghaiTech University}
  \rorid{}
  \address{No. 393 Middle Huaxia Road}
  \zip{201210}
  \city{Pudong New Area, Shanghai}
  \country{China}
  \affemail{zhangziyu@shanghaitech.edu.cn}
}

\classification[14C30]{14F08}


\begin{abstract}
    We formulate a dg-categorical analogue of the Cartan determinant conjecture using numerical Grothendieck groups. Let $\mathcal{D}$ be a smooth proper dg category over an algebraically closed field $k$. The Euler pairing on $K_0(\mathcal{D})$ descends to a nondegenerate bilinear form on the numerical Grothendieck group $K_{\mathrm{num}}(\mathcal{D})$, and hence defines a numerical Cartan determinant $\det M_\chi(\mathcal{D})$. While the classical Cartan determinant conjecture predicts that $\det M_\chi(A)=1$ for finite-dimensional smooth algebras $A$, this equality is no longer expected for general smooth proper dg categories. However, we propose the following positivity question: $\det M_\chi(\mathcal{D})>0.$

Our main result proves this positivity when $\mathcal{D}=\operatorname{Perf}(X)$ for a large class of smooth projective varieties $X$. More precisely, we show that $\det M_\chi(X)>0$ if either $\dim X$ is odd, or $X$ has even dimension and satisfies the numerical Hodge standard conjecture.

\end{abstract}

\title{Cartan Determinants in algebraic geometry}
\maketitle
\section{Introduction}
In this paper, we work over an algebraically closed field $k$. Let $\mathcal{D}$ be a smooth proper dg category over $k$. We define
$
K_0(\mathcal{D}):=K_0(\operatorname{Perf}(\mathcal{D})),
$
where $\operatorname{Perf}(\mathcal{D})$ is the triangulated category of perfect right $\mathcal{D}$-modules. Since $\mathcal{D}$ is proper, $K_0(\mathcal{D})$ is equipped with the Euler bilinear pairing
$
\chi_{\mathcal{D}}:K_0(\mathcal{D})\times K_0(\mathcal{D})\longrightarrow \mathbb Z,
$
given by
$$
\chi_{\mathcal{D}}([M],[N])
:=
\sum_{n\in \mathbb Z}(-1)^n
\dim_k \operatorname{Hom}_{\operatorname{Perf}(\mathcal{D})}(M,N[n]).
$$
Since $\mathcal{D}$ is smooth, the left and right kernels of $\chi_{\mathcal{D}}$ coincide (\cite[4.24]{tabuada2015noncommutative}). We denote the kernel by
$$
\operatorname{Ker}(\chi_{\mathcal{D}})
=
\left\{
x\in K_0(\mathcal{D})
\ \middle|\
\chi_{\mathcal{D}}(x,y)=0
\text{ for all } y\in K_0(\mathcal{D})
\right\}.
$$
The \emph{numerical Grothendieck group} of $\mathcal{D}$ is then defined as
$$
K_{\mathrm{num}}(\mathcal{D})
:=
K_0(\mathcal{D})/\operatorname{Ker}(\chi_{\mathcal{D}}).
$$
In fact, $K_{\mathrm{num}}(\mathcal{D})$ is a finitely generated free abelian group; see \cite[Section~6]{tabuada2016noncommutative} for $\operatorname{char} k=0$ and \cite[Theorem~6.1]{tabuada2019noncommutative} for $\operatorname{char} k>0$. We now define the numerical Cartan matrix of $\mathcal{D}$.
\begin{definition}
For a $\mathbb Z$-basis $e_1,\dots,e_r$ of $K_{\mathrm{num}}(\mathcal{D})$, we define the \emph{numerical Cartan matrix} of $\mathcal{D}$ with respect to $e_\bullet$ by
$$
M_\chi(e_\bullet):=\bigl(\chi(e_i,e_j)\bigr)_{i,j}.
$$
\end{definition}
If $e'_\bullet$ is another $\mathbb Z$-basis of $K_{\mathrm{num}}(\mathcal{D})$, then
$
M_\chi(e'_\bullet)=P^t M_\chi(e_\bullet)P
$
for some $P\in \mathrm{GL}_r(\mathbb Z)$. Hence
$$
\det M_\chi(e'_\bullet)
=(\det P)^2\det M_\chi(e_\bullet)
=\det M_\chi(e_\bullet),
$$
so $\det M_\chi(e_\bullet)$ is independent of the choice of basis. 
Since we are only interested in the determinant of the numerical Cartan matrix, we will suppress the choice of basis and write $M_\chi(\mathcal{D})$ for $M_\chi(e_\bullet)$.

\subsection{The classical Cartan determinant conjecture}
Let $A$ be a finite-dimensional $k$-algebra which is (cohomologically) smooth, i.e., $A$ is perfect as a right $A\otimes_k A^{\mathrm{op}}$-module. In this case,
$
K_0(A):=K_0(\operatorname{Perf} A)=K_{\mathrm{num}}(A),
$
and
$
\det(M_\chi(A)):=\operatorname{det}(M_{\chi}(\operatorname{Perf} A))=\pm 1;
$
see \cite[Proposition~21]{eilenberg1954algebras}. Zacharia conjectured that this determinant is always equal to $1$.

\begin{conj}[Cartan determinant conjecture. {\cite{zacharia1983cartan}}]\label{cartan determinat conjecture}

Let $A$ be a finite-dimensional smooth $k$-algebra. Then we have
$
\det(M_\chi(A))=1.
$
\end{conj}

Although Conjecture~\ref{cartan determinat conjecture} holds for certain finite-dimensional smooth algebras $A$ \cite{zacharia1983cartan,wilson1983cartan,burgess1985cartan,burgess1989quasi,saorin1998monoid,qin2016reducing,ingalls2020homological,green2021algebras}, it remains wide open in general.

Now let $A$ be a smooth finite-dimensional dg algebra. By \cite[Corollary~2.20]{orlov2020finite}, $\operatorname{Perf}(A)$ is an admissible subcategory of a triangulated category $\mathcal{B}$ which has a full exceptional collection. So we have
$
\det M_\chi(\mathcal{B})=1.
$
Since $\cB$ admits a semiorthogonal decomposition
$\mathcal{B}=\langle \operatorname{Perf}(A),\mathcal{C}\rangle $ for a residual category $\cC$,
we have
$
1=\det M_\chi(\mathcal{B})
=
\det M_\chi(A)\det M_\chi(\mathcal{C}).
$ Since $M_{\chi}(A)\in \ZZ$, we have
$
\det M_\chi(A)=\pm 1.
$ It is therefore natural for us to formulate the following dg analogue of Conjecture~\ref{cartan determinat conjecture}.

\begin{question}
    Let $A$ be a finite-dimensional smooth dg $k$-algebra. Is $\det(M_\chi(A)) =1$?
\end{question}

\subsection{Numerical Cartan determinant in algebraic geometry}
We may push the Cartan determinant question further.
For a general smooth proper dg category, the equality $\det M_\chi=1$ is no longer expected. Indeed, if $S$ is a K3 surface with $\operatorname{NS}(S)=\mathbb ZH$ and $H^2=2\ell$, where $\ell\geq 1$, then
$$
\det M_\chi(S):=\det M_\chi(\operatorname{Perf}(S))=2\ell\neq 1.
$$
However, observe that this determinant is still positive. This suggests the following question.

\begin{question}\label{smooth proper}
Let $\mathcal{D}$ be a smooth proper dg category over $k$. Is 
$\det M_\chi(\mathcal{D})>0?$
\end{question}

We give an affirmative answer to Question~\ref{smooth proper} in the case where $\mathcal{D}=\operatorname{Perf}(X)$, for a large class of smooth projective varieties $X$, in the following main theorem.

\begin{theorem}[{Theorem~\ref{odd case}, Theorem~\ref{even case}}]
Let $X$ be a smooth projective variety over $k$, and let $M_\chi$ be the numerical Cartan matrix of $X$. Then
$\det M_\chi(X):=\det M_\chi(\operatorname{Perf}(X))>0$
in the following cases:
\begin{enumerate}
    \item $\dim X$ is odd;
    \item $\dim X$ is even, and $X$ satisfies the numerical Hodge standard conjecture.
\end{enumerate}
\end{theorem}

In Remark~\ref{hold case}, we give a list of cases in which the numerical Hodge standard conjecture is known to hold. It is also natural to ask when the numerical Euler lattice is unimodular.

\begin{question}\label{unimodular question}
Let $X$ be a smooth projective variety over $k$. Which geometric properties of $X$ imply
$
\det M_\chi(X)=\pm 1?
$
Under what additional conditions is $\det M_\chi(X)=1$?
\end{question}

This question also gives an obstruction to the existence of numerical exceptional collections of maximal length. Indeed, if $E_1,\ldots,E_r$ is such a collection, then its Euler matrix is upper triangular with all diagonal entries equal to $1$. If $L\subset K_{\mathrm{num}}(X)$ is the sublattice generated by the classes $[E_i]$, then
$
1=[K_{\mathrm{num}}(X):L]^2\det M_\chi(X).
$
Since $\det M_\chi(X)$ is an integer, it follows that $L=K_{\mathrm{num}}(X)$ and $\det M_\chi(X)=1$. Thus $\det M_\chi(X)\neq 1$ obstructs numerical exceptional collections of maximal length. Lattice-theoretic restrictions on such collections have already been studied for surfaces, for example through the N\'eron--Severi lattice in \cite{vial2017exceptional}.

More generally, if a smooth proper dg category admits a semiorthogonal decomposition
$
\mathcal{D}=\langle \mathcal{A}_1,\ldots,\mathcal{A}_m\rangle,
$
then the numerical Cartan matrix is block upper triangular with diagonal blocks $M_\chi(\mathcal{A}_i)$. Hence
$
\det M_\chi(\mathcal{D})=\prod_{i=1}^m\det M_\chi(\mathcal{A}_i).
$
In particular, unimodularity of $K_{\mathrm{num}}(\mathcal{D})$ imposes corresponding restrictions on the numerical lattices of the components of any semiorthogonal decomposition.

\subsection{Acknowledgment}
We thank Alexander Perry for many useful discussions. 
Z.~Zhang is supported by National Natural Science Foundation of China (No.~12371046).

\section{The numerical Cartan determinant for smooth projective varieties}
Let $X$ be a smooth projective variety over $k$. Let $K_0(X)$ denote the Grothendieck group of coherent sheaves on $X$, and let $\mathrm{CH}^*(X)$ denote its Chow ring. It is well known that the Chern character defines an isomorphism
$$
\operatorname{ch}:K_0(X)_{\mathbb Q}
\xrightarrow{\sim}
\mathrm{CH}^*(X)_{\mathbb Q};
$$
see, for example, \cite[Example 15.2.16, (b)]{FultonIntersectionTheory}. Let
$$
N_K(X):=\left\{x\in K_0(X)\ \middle|\ 
\chi(x,y)=0 \text{ for all } y\in K_0(X)\right\},
$$
and
$$
N_{\mathrm{CH}}(X):=\left\{\alpha\in \mathrm{CH}^*(X)\ \middle|\ 
\int_{X}\alpha \beta=0 \text{ for all } \beta\in \mathrm{CH}^*(X)\right\}.
$$
Here $\chi(-,-)$ is the Euler pairing on $K_{0}(X)$. Recall that the numerical Grothendieck group of $X$ is
$$
K_{\mathrm{num}}(X):=K_0(X)/N_K(X),
$$
and the numerical Chow group is
$$
N^*(X):=\mathrm{CH}^*(X)/{\equiv_{\mathrm{num}}}
=\mathrm{CH}^*(X)/N_{\mathrm{CH}}(X).
$$
Since $\operatorname{Perf}(X)$ is a smooth proper dg category, $K_{\mathrm{num}}(X)$ is a finitely generated free abelian group.

For the projective case, the following lemma also appears in \cite[Lemma~2.19]{kuznetsov2018derived}.
\begin{lemma}\label{iso}
Let $X$ be a smooth projective variety over $k$. Then the Chern character
$$
\operatorname{ch}:K_0(X)_{\mathbb Q}\xrightarrow{\sim} \mathrm{CH}^*(X)_{\mathbb Q}
$$
identifies $N_K(X)_{\mathbb{Q}}$ with $N_{\mathrm{CH}}(X)_{\mathbb{Q}}$. Hence it descends to an isomorphism
$$
\overline{\operatorname{ch}}:
K_{\mathrm{num}}(X)_{\mathbb Q}
\xrightarrow{\sim}
N^{*}(X)_{\mathbb{Q}}.
$$ 
\end{lemma}

\begin{proof}
By the Hirzebruch--Riemann--Roch theorem, for $x,y\in K_0(X)_{\mathbb Q}$, we have
$$
\chi(x,y)=\int_X \operatorname{ch}(x^\vee)\operatorname{ch}(y)\operatorname{td}(X).
$$
Note that $\operatorname{td}(X)$ is invertible in $\mathrm{CH}^*(X)_{\mathbb Q}$. Therefore, as $y$ varies, the classes $\operatorname{ch}(y)\operatorname{td}(X)$
run through all of $\mathrm{CH}^*(X)_{\mathbb Q}$. Hence
$$
x\in N_K(X)_{\mathbb{Q}}
\quad\Longleftrightarrow\quad
\operatorname{ch}(x^\vee)\in N_{\mathrm{CH}}(X)_{\mathbb{Q}}.
$$

Now write
$$
\operatorname{ch}(x)=\sum_i \alpha_i,
\qquad
\alpha_i\in \mathrm{CH}^i(X)_{\mathbb Q}.
$$
Then $\operatorname{ch}(x^\vee)=\sum_i (-1)^i\alpha_i.$
The sign involution
$$
\sum_i \alpha_i\longmapsto \sum_i (-1)^i\alpha_i
$$
preserves numerical triviality. Indeed, if $d=\dim X$, then for any $\beta\in \mathrm{CH}^*(X)_{\mathbb Q}$,
$$
\int_{X}\left(\sum_i (-1)^i\alpha_i\right)\beta
=
(-1)^d\int_{X}\left(\sum_i \alpha_i\right)
\left(\sum_j (-1)^j\beta_j\right),
$$
where $\beta=\sum_j \beta_j$ with $\beta_j\in \mathrm{CH}^j(X)_{\mathbb Q}$. Since $\beta$ varies over all of $\mathrm{CH}^*(X)_{\mathbb Q}$, so does $\sum_j(-1)^j\beta_j$. Thus
$$
\operatorname{ch}(x^\vee)\in N_{\mathrm{CH}}(X)_{\mathbb{Q}}
\quad\Longleftrightarrow\quad
\operatorname{ch}(x)\in N_{\mathrm{CH}}(X)_{\mathbb{Q}}.
$$
Therefore
$$
x\in N_K(X)_{\mathbb{Q}}
\quad\Longleftrightarrow\quad
\operatorname{ch}(x)\in N_{\mathrm{CH}}(X)_{\mathbb{Q}}.
$$
So $\operatorname{ch}$ identifies the numerical radicals on both sides, and hence descends to the desired isomorphism.
\end{proof}

Let $\alpha=\sum_i \alpha_i\in N^*(X)_{\mathbb Q}$, where $\alpha_i\in N^i(X)_{\mathbb Q}$. We define
$
\alpha^\vee:=\sum_i (-1)^i\alpha_i.
$
Let $B(-,-)$ be the bilinear form on $N^*(X)_{\mathbb Q}$ defined by
$$
B(\alpha,\beta):=\int_X \alpha^\vee\beta.
$$
Suppose that $\alpha_1,\dots,\alpha_r$ is a $\mathbb Q$-basis of $N^*(X)_{\mathbb Q}$. We define
$$
M_B(\alpha_\bullet):=\bigl(B(\alpha_i,\alpha_j)\bigr)_{i,j}.
$$
If $\alpha'_\bullet$ is another $\mathbb Q$-basis of $N^*(X)_{\mathbb Q}$, then
$$
M_B(\alpha'_\bullet)=P^t M_B(\alpha_\bullet)P
$$
for some $P\in \mathrm{GL}_r(\mathbb Q)$. Hence
$$
\det M_B(\alpha'_\bullet)
=(\det P)^2\det M_B(\alpha_\bullet).
$$
Since $(\det P)^2>0$, the sign of $\det M_B(\alpha_\bullet)$ is independent of the choice of basis. In this paper, we are only interested in the sign of this determinant.  So we will simply write $M_B(X)$ for $M_B(\alpha_\bullet)$.

\begin{lemma}\label{these two are the same}
Let
$$
\overline{\operatorname{ch}}:
K_{\mathrm{num}}(X)_{\mathbb Q}
\xrightarrow{\sim}
N^*(X)_{\mathbb Q}
$$
be the isomorphism in Lemma~\ref{iso}. Let $e_\bullet=(e_1,\dots,e_r)$ be a $\mathbb Z$-basis of $K_{\mathrm{num}}(X)$, and set
$
\alpha_i:=\overline{\operatorname{ch}}(e_i).
$
Then $\alpha_\bullet=(\alpha_1,\dots,\alpha_r)$ is a $\mathbb Q$-basis of $N^*(X)_{\mathbb Q}$, and
$$
\det M_\chi(e_\bullet)=\det M_B(\alpha_\bullet).
$$
In particular,
$
\operatorname{sgn}\det M_\chi(X)=\operatorname{sgn}\det M_B(X).
$
\end{lemma}

\begin{proof}
By the Hirzebruch--Riemann--Roch theorem,  we have
$$
\chi(e_{i},e_{j})=\int_X \alpha_{i}^\vee\alpha_{j}\operatorname{td}(X).
$$
Let
$
T:N^*(X)_{\mathbb Q}\longrightarrow N^*(X)_{\mathbb Q}
$
be multiplication by $\operatorname{td}(X)$.
Let $M_T(\alpha_\bullet)$ be the matrix of $T$ in this basis. Since
$
\chi(e_i,e_j)=B(\alpha_i,T\alpha_j),
$
we get
$$
M_{\chi}(e_\bullet)=M_B(\alpha_\bullet)M_T(\alpha_\bullet).
$$
Taking determinants gives
$$
\det M_{\chi}(e_\bullet)
=
\det M_B(\alpha_\bullet)\det M_T(\alpha_\bullet).
$$

We now show that $\det M_{T}=1.$
Write
$$
\operatorname{td}(X)=1+\operatorname{td}_1(X)+\cdots+\operatorname{td}_d(X),
\qquad
\operatorname{td}_j(X)\in N^j(X)_{\mathbb Q}.
$$
Choose a homogeneous basis of $N^*(X)_{\mathbb Q}$. 
For $\gamma\in N^i(X)_{\mathbb Q}$, we have
$$
T(\gamma)=\gamma\operatorname{td}(X)
=\gamma+\gamma\operatorname{td}_1(X)+\cdots+\gamma\operatorname{td}_{d-i}(X),
$$
where
$
\gamma\operatorname{td}_j(X)\in N^{i+j}(X)_{\mathbb Q}.
$
Thus $T$ preserves the codimension filtration and induces the identity on each graded piece $N^i(X)_{\mathbb Q}$. Hence, in such a basis, the matrix of $T$ is triangular with identity diagonal blocks. 
Since the determinant of a linear operator is independent of the choice of basis, $\operatorname{det}M_{T}(\alpha^{\bullet})=1$. So we have $$\det M_{\chi}(e_\bullet)
=
\det M_B(\alpha_\bullet)\det M_T(\alpha_\bullet)=\operatorname{det}M_{B}(\alpha_{\bullet}). $$
This proves the lemma.
\end{proof}

\begin{proposition}\label{perfect pairing}
Let $d=\dim X$. For $i<d/2$, choose bases of $N^i(X)_{\mathbb Q}$ and $N^{d-i}(X)_{\mathbb Q}$, and let $P_i$ be the matrix of the perfect pairing
$$
N^i(X)_{\mathbb Q}\times N^{d-i}(X)_{\mathbb Q}
\longrightarrow \mathbb Q,
\qquad
(\alpha,\beta)\longmapsto \int_X \alpha\beta .
$$
Then the block of $M_B$ on $N^i(X)_{\mathbb Q}\oplus N^{d-i}(X)_{\mathbb Q}
$
is
$$
\begin{pmatrix}
0 & (-1)^iP_i \\
(-1)^{d-i}P_i^t & 0
\end{pmatrix}.
$$
In particular, if $r_i=\dim_{\mathbb Q}N^i(X)_{\mathbb Q}$, then this block has determinant
$$
(-1)^{(d+1)r_i}\det(P_i)^2.
$$
\end{proposition}

\begin{proof}
For
$$
\alpha\in N^i(X)_{\mathbb Q},
\qquad
\beta\in N^j(X)_{\mathbb Q},
$$
we have
$$
B(\alpha,\beta)
=
\int_X \alpha^\vee\beta
=
(-1)^i\int_X \alpha\beta .
$$
This is zero unless $i+j=d$. Hence $B$ pairs $N^i(X)_{\mathbb Q}$ only with $N^{d-i}(X)_{\mathbb Q}$. Therefore the off-diagonal blocks are
$$
(-1)^iP_i
\qquad\text{and}\qquad
(-1)^{d-i}P_i^t,
$$
which gives the claimed matrix.

\end{proof}

We now prove that $\det M_{\chi}(X)>0$ for every odd-dimensional variety.
\begin{theorem}\label{odd case}
Suppose that $d=\dim X$ is odd. Then
$
\det M_\chi(X)>0.
$
\end{theorem}

\begin{proof}
Write
$
d=2m+1.
$
Choose a homogeneous basis of $N^*(X)_{\mathbb Q}$, ordered according to the decomposition
$$
N^*(X)_{\mathbb Q}
=
\bigoplus_{i=0}^m
\left(N^i(X)_{\mathbb Q}\oplus N^{d-i}(X)_{\mathbb Q}\right).
$$
Since $d$ is odd, there is no middle term. By Proposition~\ref{perfect pairing}, the block of $M_B$ on
$$
N^i(X)_{\mathbb Q}\oplus N^{d-i}(X)_{\mathbb Q}
$$
has determinant
$$
(-1)^{(d+1)r_i}\det(P_i)^2,
\qquad
r_i=\dim_{\mathbb Q}N^i(X)_{\mathbb Q}.
$$
But
$
d+1=2m+2
$
is even, hence
$
(-1)^{(d+1)r_i}=1.
$
Therefore every block has positive determinant:
$$
(-1)^{(d+1)r_i}\det(P_i)^2
=
\det(P_i)^2>0.
$$
Thus
$
\det M_B>0.
$
By Lemma~\ref{these two are the same}, we have
$$
\operatorname{sgn}\det M_\chi(X)
=
\operatorname{sgn}\det M_B(X).
$$
Hence $\det M_\chi(X)>0.$
This proves the theorem.
\end{proof}

\begin{remark}
There is also a conceptual way to understand Theorem~\ref{odd case}. When $\dim X$ is odd, the bilinear form $B$ is skew-symmetric. Hence the determinant of a matrix $M_B$ representing $B$ is the square of its Pfaffian, and is therefore non-negative. Since $B$ is nondegenerate, its determinant is nonzero, and hence positive. Together with Lemma~\ref{these two are the same}, this gives another proof of the positivity of $\det M_\chi(X)$.
\end{remark}

When $\dim X$ is even, we cannot directly show that $\det M_{\chi}(X)>0$ as in Theorem~\ref{odd case}. To deal with this case, we first recall the numerical Hodge standard conjecture.

\paragraph{Numerical Hodge standard conjecture.}\label{numerical}
 Let $X$ be a smooth projective variety over $k$ of dimension $\dim X=d,$
and let $L\in N^1(X)_{\mathbb Q}$
be the class of an ample divisor. For $0\leq i\leq d/2$, set
$$
P^i_{\mathrm{num}}(X,L)
:=
\ker\left(
L^{d-2i+1}:N^i(X)_{\mathbb Q}
\longrightarrow
N^{d-i+1}(X)_{\mathbb Q}
\right).
$$
We say that $X$ satisfies the \emph{numerical Hodge standard conjecture} if, for every $0\leq i\leq d/2$, the form
$$
Q_i(\alpha,\beta):=
(-1)^i\int_X L^{d-2i}\alpha\beta
$$
is positive definite on
$P^i_{\mathrm{num}}(X,L).$

This is the standard conjecture of Hodge type, formulated for numerical cycle groups; 
see Grothendieck's original paper on the standard conjectures
\cite{Grothendieck1969StandardConjectures}, Kleiman's survey
\cite{steven1994standard}, and Milne's discussion of polarizations and the Hodge standard conjecture
\cite{milne2002polarizations}.

\begin{proposition}\label{Hodge}
Assume that $X$ satisfies the numerical Hodge standard conjecture. Then numerical hard Lefschetz holds: for $0\leq i\leq d/2$,
$$
L^{d-2i}:N^i(X)_{\mathbb Q}\xrightarrow{\sim}N^{d-i}(X)_{\mathbb Q}.
$$
Consequently, we have the numerical Lefschetz decomposition
$$
N^i(X)_{\mathbb Q}
=
LN^{i-1}(X)_{\mathbb Q}\oplus P^i_{\mathrm{num}}(X,L).
$$

\end{proposition}

\begin{proof}
First prove injectivity of $L^{d-2i}.$
Suppose
$
L^{d-2i}\alpha=0.
$
Then automatically
$
L^{d-2i+1}\alpha=0,
$
so $\alpha\in P^i_{\mathrm{num}}(X,L)$. Hence
$$
Q_i(\alpha,\alpha)
=
(-1)^i\int_X L^{d-2i}\alpha^2
=
0.
$$
By positive definiteness on $P^i_{\mathrm{num}}(X,L)$, this implies $\alpha=0$. Thus the map is injective. The source and target have the same dimension by the perfect numerical intersection pairing, so the map is an isomorphism.

Next we prove the Lefschetz decomposition. For $0\leq i\leq d/2$, we claim
$$
N^i(X)_{\mathbb Q}
=
LN^{i-1}(X)_{\mathbb Q}\oplus P^i_{\mathrm{num}}(X,L),
$$
where $N^{-1}(X)_{\mathbb Q}=0$. The intersection is zero. Indeed, if
$
L\gamma\in P^i_{\mathrm{num}}(X,L),
$
then
$$
0=L^{d-2i+1}(L\gamma)=L^{d-2i+2}\gamma.
$$
But
$$
L^{d-2i+2}:N^{i-1}(X)_{\mathbb Q}\longrightarrow N^{d-i+1}(X)_{\mathbb Q}
$$
is an isomorphism by numerical hard Lefschetz, so $\gamma=0$.

It remains to check dimensions.  Since
$$
L^{d-2i+2}=L^{d-2i+1}\circ L:N^{i-1}(X)_{\mathbb Q}\xrightarrow{\sim}N^{d-i+1}(X)_{\mathbb Q}
$$
is an isomorphism, we have $L: N^{i-1}(X)_{\mathbb{Q}}\to N^{i}(X)_{\mathbb{Q}}$ is injective and $L^{d-2i+1}: N^i(X)_{\mathbb Q}\longrightarrow N^{d-i+1}(X)_{\mathbb Q}$ is surjective. Therefore $
\dim_{\mathbb Q} LN^{i-1}(X)_{\mathbb Q}=r_{i-1}
$, and 
$$
\dim P^i_{\mathrm{num}}(X,L)=\operatorname{dim}N^{i}(X)_{\mathbb{Q}}-\operatorname{dim}N^{d-i+1}(X)_{\mathbb{Q}}= r_i-r_{d-i+1}=r_{i}-r_{i-1}.
$$
 Thus
$$
N^i(X)_{\mathbb Q}
=
LN^{i-1}(X)_{\mathbb Q}\oplus P^i_{\mathrm{num}}(X,L).
$$

\end{proof}

We can now prove that $\det M_{\chi}(X)>0$ for even-dimensional $X$ under the assumption that $X$ satisfies the numerical Hodge standard conjecture.

\begin{theorem}\label{even case}
Let $X$ be a smooth projective variety of dimension $2m$. Suppose that $X$ satisfies the numerical Hodge standard conjecture. Then
$ \det M_\chi(X)>0.$
\end{theorem}

\begin{proof}
By Lemma~\ref{these two are the same}, it is enough to prove that
$
\det M_B>0.
$

First consider the non-middle blocks. For $0\leq i\leq m-1$, Proposition~\ref{perfect pairing} gives the block on
$
N^i(X)_{\mathbb Q}\oplus N^{2m-i}(X)_{\mathbb Q}
$
with determinant sign
$$
(-1)^{(2m+1)r_i}=(-1)^{r_i}.
$$
Therefore the product of the signs of all non-middle blocks is
$
(-1)^{r_0+\cdots+r_{m-1}}.
$

It remains to compute the sign of the middle block on
$ N^m(X)_{\mathbb Q}.$
On this space,
$$
B(\alpha,\beta)=(-1)^m\int_X \alpha\beta.
$$
Iterating the numerical Lefschetz decomposition at $i=m$ in Proposition~\ref{Hodge} gives
$$
N^m(X)_{\mathbb Q}
=
\bigoplus_{i=0}^m L^{m-i}P^i_{\mathrm{num}}(X,L).
$$
Let
$
p_i:=\dim_{\mathbb Q}P^i_{\mathrm{num}}(X,L),$ then
$
p_i=r_i-r_{i-1}.
$

The Lefschetz summands are mutually orthogonal with respect to $B$. Indeed, suppose that  $i<j$, $\alpha\in P^i_{\mathrm{num}}(X,L)$, and $\beta\in P^j_{\mathrm{num}}(X,L)$. Since
$
\beta\in P^j_{\mathrm{num}}(X,L),
$
we have
$
L^{2m-2j+1}\beta=0.
$
Moreover,
$
2m-i-j=(2m-2j+1)+(j-i-1),
$
and $j-i-1\geq 0$. Therefore
$$
L^{2m-i-j}\alpha\beta
=
(L^{2m-2j+1}\beta)(L^{j-i-1}\alpha)
=
0.
$$
Hence
$$
B(L^{m-i}\alpha,L^{m-j}\beta)
=
(-1)^m\int_X L^{2m-i-j}\alpha\beta
=
0.
$$

On the summand $L^{m-i}P^i_{\mathrm{num}}(X,L)$, we have
$$
B(L^{m-i}\alpha,L^{m-i}\alpha)
=
(-1)^m\int_X L^{2m-2i}\alpha^2=
(-1)^{m-i}
\left((-1)^i\int_X L^{2m-2i}\alpha^2\right) .
$$
The expression in parentheses is
$
Q_i(\alpha,\alpha),
$
which is positive definite on $P^i_{\mathrm{num}}(X,L)$. Hence the summand $L^{m-i}P^i_{\mathrm{num}}(X,L)$ is positive definite for $B$ if $m-i$ is even, and negative definite if $m-i$ is odd.

Therefore the number of negative directions of the middle block is
$$
s=\sum_{\substack{0\leq i\leq m\\ m-i\ \mathrm{odd}}}p_i.
$$
Using
$
p_i=r_i-r_{i-1},
$
we get, modulo $2$,
$
s\equiv r_0+r_1+\cdots+r_{m-1}.
$
Thus the sign of the determinant of the middle block is
$$
(-1)^s
=
(-1)^{r_0+\cdots+r_{m-1}}.
$$

This is exactly the same sign as the product of the non-middle blocks. Therefore the total sign of $\det M_B(X)$ is
$$
(-1)^{r_0+\cdots+r_{m-1}}
\cdot
(-1)^{r_0+\cdots+r_{m-1}}
=
+1.
$$
 This implies
$
\det M_B(X)>0.
$
This proves the even-dimensional case.
\end{proof}

\begin{remark}\label{hold case}
We have the following known cases.
\begin{itemize}
    \item When $\dim X=2$, the numerical Hodge standard conjecture is just the Hodge index theorem, which is known. 
    \item When $\operatorname{Perf}(X)$ admits a full exceptional collection, we have $\det M_{\chi}(X)=1$.
    \item The numerical Hodge standard conjecture also holds in the following cases.

In characteristic zero:
\begin{enumerate}
    \item The numerical Hodge standard conjecture is stable under products and hyperplane sections; see \cite[Proposition 4.3.(1)]{steven1994standard}.

    \item $\dim X\leq 4$; see \cite[Corollary 1]{lieberman1968numerical}.

    \item $X$ is an abelian variety; see \cite[Theorem 4]{lieberman1968numerical}.

    \item $X$ is a smooth projective toric variety; see \cite[Page 106]{fulton1993introduction}.

    \item $X$ is a generalized flag variety; see \cite[Proposition 4.3.(2)]{steven1994standard}.

    \item $X$ is motivated by a curve or a surface. In particular, this applies when $X$ is a uniruled threefold or a unirational fourfold; see \cite[Corollary 4.3]{arapura2006motivation}.

    \item $X=S^{[n]}$, the Hilbert scheme of $n$ points on a smooth projective surface $S$; see \cite[Theorem 6.2.1]{de2002chow} and \cite[Corollary 4.3]{arapura2006motivation}.

    \item $X$ is a smooth projective variety of $K3^{[n]}$-type; see \cite[Corollary 1.2]{charles2013standard}.

    \item $X$ is a certain hyperk\"ahler manifold of O'Grady's 10-dimensional deformation type; see \cite[Corollary 1.9]{floccari2021motive}.

    \item $X$ is a projective irreducible holomorphic symplectic manifold of generalized Kummer deformation type of dimension $2n$, where $n+1$ is prime; see \cite[Corollary 1.2]{Foster2024KummerLSC}.
\end{enumerate}

In positive characteristic, fewer cases are known:
\begin{enumerate}
    \item $X$ is a smooth projective toric variety; see \cite[Theorem~8.2]{mcmullen1993simple}.

    \item $X$ is an abelian fourfold; see \cite[Theorem~1.3]{giuseppe2021standard}.

    \item $X=A\times A$ or $X=A\times E$ for an elliptic curve $E$ and certain abelian varieties $A$; see \cite[Theorem~1.1]{koshikawa2024numerical}.

    \item $X=S\times S$ for a K3 surface $S$; see \cite[Corollary~1.4]{ito2025hodge}. More generally, $X=S^{\times n}$ for a K3 surface $S$ with $\rho(S)\geq 17$; see \cite[Theorem~1.9]{ito2025hodge}.

    \item $X$ belongs to a certain class of  abelian varieties over $\overline{\mathbb F}_p$; see \cite[Theorem~1.1]{agugliaro2026examples}.
\end{enumerate}

\end{itemize}

\end{remark}

\bibliographystyle{emss} 
\bibliography{ref}

@article{orlov2020finite,
  title={Finite-dimensional differential graded algebras and their geometric realizations},
  author={Orlov, Dmitri},
  journal={Advances in Mathematics},
  volume={366},
  pages={107096},
  year={2020},
  publisher={Elsevier}
}

@article{eilenberg1954algebras,
  title={Algebras of cohomologically finite dimension},
  author={Eilenberg, Samuel},
  journal={Commentarii Mathematici Helvetici},
  volume={28},
  number={1},
  pages={310--319},
  year={1954},
  publisher={Springer}
}

@book{FultonIntersectionTheory,
  author    = {Fulton, William},
  title     = {Intersection Theory},
  edition   = {2},
  series    = {Ergebnisse der Mathematik und ihrer Grenzgebiete. 3. Folge},
  volume    = {2},
  publisher = {Springer},
  address   = {Berlin},
  year      = {1998},
  isbn      = {978-0-387-98549-7},
  doi       = {10.1007/978-1-4612-1700-8}
}

@article{milne2002polarizations,
  title={Polarizations and {Grothendieck}'s standard conjectures},
  author={Milne, James Stuart},
  journal={Annals of mathematics},
  pages={599--610},
  year={2002},
  publisher={JSTOR}
}

@incollection{Grothendieck1969StandardConjectures,
  author    = {Grothendieck, Alexander},
  title     = {Standard conjectures on algebraic cycles},
  booktitle = {Algebraic Geometry (Internat. Colloq., Tata Inst. Fund. Res., Bombay, 1968)},
  pages     = {193--199},
  publisher = {Oxford University Press},
  address   = {London},
  year      = {1969}
}

@article{lieberman1968numerical,
  title={Numerical and homological equivalence of algebraic cycles on {Hodge} manifolds},
  author={Lieberman, David I},
  journal={American Journal of Mathematics},
  volume={90},
  number={2},
  pages={366--374},
  year={1968},
  publisher={JSTOR}
}

@incollection{steven1994standard,
  author    = {Kleiman, Steven L.},
  title     = {The Standard Conjectures},
  booktitle = {Motives},
  series    = {Proceedings of Symposia in Pure Mathematics},
  volume    = {55},
  part      = {1},
  pages     = {3--20},
  publisher = {American Mathematical Society},
  address   = {Providence, RI},
  year      = {1994},
  note      = {Seattle, WA, 1991}
}

@book{fulton1993introduction,
  author    = {Fulton, William},
  title     = {Introduction to Toric Varieties},
  series    = {Annals of Mathematics Studies},
  volume    = {131},
  publisher = {Princeton University Press},
  address   = {Princeton, NJ},
  year      = {1993}
}

@article{de2002chow,
  author  = {de Cataldo, Mark Andrea A. and Migliorini, Luca},
  title   = {The {Chow} Groups and the Motive of the {Hilbert} Scheme of Points on a Surface},
  journal = {Journal of Algebra},
  volume  = {251},
  pages   = {824--848},
  year    = {2002},
  doi     = {10.1006/jabr.2001.9105}
}

@article{arapura2006motivation,
  title={Motivation for Hodge cycles},
  author={Arapura, Donu},
  journal={Advances in Mathematics},
  volume={207},
  number={2},
  pages={762--781},
  year={2006},
  publisher={Elsevier}
}

@article{charles2013standard,
  title={The standard conjectures for holomorphic symplectic varieties deformation equivalent to {Hilbert} schemes of {K}3 surfaces},
  author={Charles, Fran{\c{c}}ois and Markman, Eyal},
  journal={Compositio Mathematica},
  volume={149},
  number={3},
  pages={481--494},
  year={2013},
  publisher={London Mathematical Society}
}

@article{Foster2024KummerLSC,
  author       = {Foster, Josiah},
  title        = {The {Lefschetz} standard conjectures for {IHSMs} of generalized {Kummer} deformation type in certain degrees},
  journal      = {European Journal of Mathematics},
  volume       = {10},
  year         = {2024},
  article      = {34},
  doi          = {10.1007/s40879-024-00744-2}
}

@article{floccari2021motive,
  title={On the motive of {O'Grady}'s ten-dimensional hyper-{K{\"a}hler} varieties},
  author={Floccari, Salvatore and Fu, Lie and Zhang, Ziyu},
  journal={Communications in Contemporary Mathematics},
  volume={23},
  number={04},
  pages={2050034},
  year={2021},
  publisher={World Scientific}
}

@article{giuseppe2021standard,
  author  = {Ancona, Giuseppe},
  title   = {Standard Conjectures for Abelian Fourfolds},
  journal = {Inventiones Mathematicae},
  volume  = {223},
  pages   = {149--212},
  year    = {2021},
  doi     = {10.1007/s00222-020-00990-7}
}

@article{mcmullen1993simple,
  title={On simple polytopes},
  author={McMullen, Peter},
  journal={Inventiones mathematicae},
  volume={113},
  number={1},
  pages={419--444},
  year={1993},
  publisher={Springer}
}

@article{koshikawa2024numerical,
  author  = {Koshikawa, Teruhisa},
  title   = {The Numerical {Hodge} Standard Conjecture for the Square of a Simple Abelian Variety of Prime Dimension},
  journal = {Manuscripta Mathematica},
  volume  = {173},
  pages   = {1161--1169},
  year    = {2024},
  doi     = {10.1007/s00229-023-01482-7}
}

@article{ito2025hodge,
  title={The {Hodge} standard conjecture for self-products of {K}3 surfaces},
  author={Ito, Kazuhiro and Ito, Tetsushi and Koshikawa, Teruhisa},
  journal={Journal of Algebraic Geometry},
  volume={34},
  number={2},
  pages={299--330},
  year={2025}
}

@article{agugliaro2026examples,
  title={Examples for the standard conjecture of Hodge type},
  author={Agugliaro, Thomas},
  journal={Documenta Mathematica},
  year={2026}
}

@book{tabuada2015noncommutative,
  author    = {Tabuada, Gon{\c{c}}alo},
  title     = {Noncommutative Motives},
  series    = {University Lecture Series},
  volume    = {63},
  publisher = {American Mathematical Society},
  address   = {Providence, RI},
  year      = {2015},
  pages     = {114},
  isbn      = {978-1-4704-2397-1},
  doi       = {10.1090/ulect/063},
  note      = {With a preface by Yuri I. Manin}
}

@article{tabuada2016noncommutative,
  title={Noncommutative motives of separable algebras},
  author={Tabuada, Gon{\c{c}}alo and Van den Bergh, Michel},
  journal={Advances in Mathematics},
  volume={303},
  pages={1122--1161},
  year={2016},
  publisher={Elsevier}
}

@article{tabuada2019noncommutative,
  title={Noncommutative motives in positive characteristic and their applications},
  author={Tabuada, Gon{\c{c}}alo},
  journal={Advances in mathematics},
  volume={349},
  pages={648--681},
  year={2019},
  publisher={Elsevier}
}

@article{zacharia1983cartan,
  title={On the {Cartan} matrix of an {Artin} algebra of global dimension two},
  author={Zacharia, Dan},
  journal={Journal of Algebra},
  volume={82},
  number={2},
  pages={353--357},
  year={1983},
  publisher={Academic Press}
}

@article{wilson1983cartan,
author  = {Wilson, G. V.},
title   = {The {Cartan} map on categories of graded modules},
journal = {Journal of Algebra},
volume  = {85},
number  = {2},
year    = {1983},
pages   = {390--398}
}

@article{burgess1989quasi,
author  = {Burgess, W. D. and Fuller, K. R.},
title   = {On quasi-hereditary rings},
journal = {Proceedings of the American Mathematical Society},
volume  = {106},
number  = {2},
year    = {1989},
pages   = {321--328}
}

@article{burgess1985cartan,
author  = {Burgess, W. D. and Fuller, K. R. and Voss, E. R. and Zimmermann-Huisgen, B.},
title   = {The {Cartan} matrix as an indicator of finite global dimension for artinian rings},
journal = {Proceedings of the American Mathematical Society},
volume  = {95},
number  = {2},
year    = {1985},
pages   = {157--165}
}

@article{saorin1998monoid,
author  = {Saor{\'i}n, M.},
title   = {Monoid gradings on algebras and the {Cartan} determinant conjecture},
journal = {Proceedings of the Edinburgh Mathematical Society},
volume  = {41},
number  = {3},
year    = {1998},
pages   = {539--551}
}

@article{qin2016reducing,
  title={Reducing homological conjectures by n-recollements},
  author={Qin, Yongyun and Han, Yang},
  journal={Algebras and Representation Theory},
  volume={19},
  number={2},
  pages={377--395},
  year={2016},
  publisher={Springer}
}

@article{green2021algebras,
  title={Algebras and varieties},
  author={Green, Edward L and Hille, Lutz and Schroll, Sibylle},
  journal={Algebras and Representation Theory},
  volume={24},
  number={2},
  pages={367--388},
  year={2021},
  publisher={Springer}
}

@article{ingalls2020homological,
  title={Homological behavior of idempotent subalgebras and {Ext} algebras},
  author={Ingalls, Colin and Paquette, Charles},
  journal={Science China Mathematics},
  volume={63},
  number={2},
  pages={309--320},
  year={2020},
  publisher={Springer}
}

@article{kuznetsov2018derived,
  author  = {Kuznetsov, Alexander and Perry, Alexander},
  title   = {Derived categories of {Gushel--Mukai} varieties},
  journal = {Compositio Mathematica},
  volume  = {154},
  number  = {7},
  pages   = {1362--1406},
  year    = {2018},
  doi     = {10.1112/S0010437X18007091}
}

@article{vial2017exceptional,
  author  = {Vial, Charles},
  title   = {Exceptional collections, and the {N}{\'e}ron--{S}everi lattice for surfaces},
  journal = {Advances in Mathematics},
  volume  = {305},
  pages   = {895--934},
  year    = {2017},
  doi     = {10.1016/j.aim.2016.10.012}
}

\end{document}